\documentclass{amsart}
\usepackage[colorlinks, linkcolor=blue,  citecolor=blue, urlcolor=blue]{hyperref}%

\usepackage[utf8]{inputenc}
\usepackage[english]{babel}
\usepackage{amsmath,amsfonts,amssymb,empheq}

\usepackage{url}
\usepackage{xspace}
\usepackage{graphicx}
\usepackage{xargs}                      
\usepackage[pdftex,dvipsnames]{xcolor}  
\usepackage[export]{adjustbox}

\usepackage[linesnumbered,ruled]{algorithm2e}

\usepackage{siunitx}
\usepackage{todonotes}
\newcommandx{\note}[2][1=]{\todo[linecolor=blue,backgroundcolor=blue!25,bordercolor=blue,#1]{#2}}

\usepackage{adjustbox}
\usepackage{pgfplots}
\usetikzlibrary{arrows}

\definecolor{cbred}{RGB}{228,26,28}
\definecolor{cbblue}{RGB}{55,126,184}
\definecolor{cbgreen}{RGB}{77,175,74}
\definecolor{cbpurple}{RGB}{152,78,163}

\usepackage{caption}
\usepackage{subcaption}

\newcommand{\R}{\mathbb{R}}

\newcommand{\w}{w}
\newcommand{\ww}{{\w \w}}

\newcommand{\y}{y}

\newcommand{\trp}{\mathsf{T}}
\newcommand{\norm}[1]{\left\lVert#1\right\rVert}

\newtheorem{theorem}{Theorem}

\DeclareMathOperator{\dist}{dist}

\DeclareMathOperator{\cosec}{cosec}

\theoremstyle{remark}
\newtheorem*{remark}{Remark}
\newtheorem{definition}{Definition}

\newtheorem{lemma}{Lemma}

\begin{document}

\title[Finite difference schemes]{Improved accuracy of monotone finite difference schemes on point clouds and regular grids}
\author[Finlay and Oberman]{Chris Finlay \and Adam Oberman}

\begin{abstract}
  Finite difference schemes are the method of choice for solving nonlinear, degenerate elliptic PDEs, because the Barles-Sougandis convergence framework \cite{BSNum} provides sufficient conditions for convergence to the unique viscosity solution \cite{CIL}.
  For anisotropic operators, such as the Monge-Ampere equation, wide stencil schemes are needed \cite{ObermanDiffSchemes}.
  The accuracy of these schemes depends on both the distances to neighbors, $R$,
  and the angular resolution, $d\theta$.  On uniform grids, the accuracy is
  $\mathcal O(R^2 + d\theta)$.   On point clouds, the most accurate schemes are of
  $\mathcal O(R + d\theta)$, by  Froese~\cite{froese_meshfree_2017}.  In this
  work, we construct geometrically motivated schemes of higher accuracy in both
  cases:  order $\mathcal O(R + d\theta^2)$ on point clouds, and $\mathcal O(R^2 + d\theta^2)$ on uniform grids.
\end{abstract}

\date{\today}
\maketitle

\section{Introduction}

The goal of this paper is to build more accurate convergent discretizations for
the class of nonlinear elliptic partial differential equations \cite{CIL}.  Our
schemes are are implemented in both two and three dimensions for a class of
PDEs, which include the convex envelope operator and the Pucci operator, as well
as the Monge-Ampere operator. Convergent discretizations for these operators are
available on uniform grids \cite{oberman_wide_2008}, but the accuracy of these
schemes depends on both the distances to neighbors, $R$, and the angular
resolution, $d\theta$.  On uniform grids, the accuracy is  $\mathcal O(R^2 +
d\theta)$.    More recently, \cite{froese_meshfree_2017} developed methods on
point clouds of accuracy $\mathcal O(R + d\theta)$.   These schemes were used
for freeform optical design to shape laser beams~\cite{Feng:17}, an application
which required nonuniform grids.  In this work, we construct geometrically
motivated schemes of higher accuracy in both cases:  order $\mathcal O(R +
d\theta^2)$ on point clouds, and $\mathcal O(R^2 + d\theta^2)$ on uniform grids.

Even higher accuracy is possible when the operator is uniformly elliptic.  For
example, in the set of papers
\cite{benamou_monotone_2014,fehrenbach_sparse_2014,mirebeau_anisotropic_2014,mirebeau_minimal_2014},
Mirebeau and coauthors developed a framework for constructing $\mathcal O(h^2)$
monotone and stable schemes for several functions of the eigenvalues of the
Hessian on uniform grids, in two dimensions. Related work for discretization of
convex functions is studied in~\cite{mirebeau2016adaptive}.   Mirebeau studied
monotone discretization of first order (Eikonal type) equations on triangulated
grids \cite{mirebeau_anisotropic_2014} as well as second order Monge-Ampere type
operators~\cite{mirebeau_minimal_2014}.  In the latter case, he obtains nearly
optimal accuracy, but his construction is most effective when the operator is
uniformly elliptic: as the operator degenerates, the width of the stencil
increases.  Moreover, the elegant construction based on the Stern-Brocot tree is
particular to two dimensions.

Higher accuracy is also possible using filtered schemes
\cite{froese2013convergent, oberman2015filtered, bokanowski2016high} Filtered
schemes combine a base monotone scheme with a higher accuracy schemes: however
increases accuracy of the base scheme is beneficial to the filtered scheme,
since it allows for a smaller filter parameter.

The challenge of building monotone convergent finite difference schemes is
illustrated in \cite{chen2016monotone} and \cite{chen2017multigrid}, 
discretizing the Monge-Ampere equation in two dimensions. In \cite{chen2016monotone}, a mixture of
a 7-point stencil for the cross and a semi-Lagrangian wide stencil was used. The 7-point
stencil was used for the cross derivative when it is monotone; otherwise the wide
stencil was employed. This approach was later extended to a multigrid in
\cite{chen2017multigrid}, but does not fully solve the problem of building
narrow monotone stencils, and has not been generalized to higher dimensions.

Another approach lies between the wide stencil finite difference approach, and
the finite element approach.  In \cite{nochetto2017two} a convergent method on
an unstructured mesh is constructed on two separate scales.  They prove a rate
of convergence (which is stronger than our results, which concern the accuracy
of the discretization).  However, there is a large gap between the rate of
convergence, and the accuracy, which is more consistent with computational
results.  For a recent review, see \cite{neilan_salgado_zhang_2017}.

The need for wide stencils arise from the anisotropy of the operators.  For
isotropic operators, such as the Laplacian, or for operators whose second order
anisotropy happens to align with the grid (essentially combinations of $u_{xx}$
and $u_{yy}$ terms) an adaptive quadtree grid discretization was developed
in~\cite{oberman2016adaptive}. An adaptive quadtree grid was combined with the
$\mathcal O(R+d\theta^2)$ meshfree method of Froese \cite{froese_meshfree_2017}
and filtered schemes \cite{oberman2015filtered, froese2013convergent}  in
\cite{froese_higher2017}.

The main idea of this work is based on locating the reference point within two
triangles (in two dimensions) or simplices (in three or higher dimensions), and
using barycentric coordinates \cite[\S 5.4 p.595]{dahlquist2008numerical} to
write down the discretization.  For first order derivates, only one simplex is
needed.  It is standard to write a gradient of a function based on linear
interpolation, extending this to a directional derivative amounts to computing a
dot product.  However, for second directional derivatives, it is possible to use
two simplices to compute a monotone discretization of the second directional
derivative, with accuracy which depends on the relative sizes of the simplices.

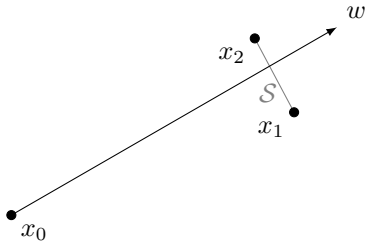
\begin{figure}
  \centering
  \trimbox{-1cm -.5cm -1cm 0cm}{\begin{tikzpicture}[scale=2]
  \draw[fill=black] (0,0) circle (0.03) node [below right] {$ x_0$};
  \draw[gray] (20:2) -- (36:2) node [pos=0.4, below left,inner sep=0.2] {$\mathcal S$};
  \draw[fill=black] (20:2) circle (0.03) node [below left] {$ x_1$};
  \draw[fill=black] (36:2) circle (0.03) node [below left] {$ x_2$};
  \draw[->, >=latex] (30:0) -- (30:2.5) node [above right] {$ w$};
\end{tikzpicture}}
  \caption{A stencil for a first derivative at $ x_0$ in the direction of
    $ w$ uses barycentric coordinates of the simplex $\mathcal S$ formed by  $ x_1$
  and $ x_2$.}
  \label{fig:introfig}
\end{figure}

\subsection{Off-directional discretizations}
When the direction $\w$ does not align with the grid, the $d\theta$ term appears
in the expression for the finite difference accuracy.  If $u$ is discretized on
a regular grid, then one common approach is to choose the nearest grid direction
$ v_h$ to $\w$, and take the finite difference along this approximate
direction, as in \cite{oberman_wide_2008}. In the symmetric case for the second
derivative, the finite difference remains $\mathcal O(h^2)$, but picks up a
directional resolution error $d \theta$. This directional resolution error is
first order, and is given as $d \theta = \arccos \langle  w,  v_h/\norm{
v_h} \rangle$.  Overall this approach is $\mathcal O(d\theta + R^2)$ accurate.
On a grid with spatial resolution $h$, one can show that for a desired angular
resolution $d\theta$, $R$ is $\mathcal O(\frac{h}{d\theta})$. With optimal
choice $d\theta = \left( 2 h^2 \right)^\frac{1}{3}$, this scheme is therefore
formally $\mathcal O(h^\frac{2}{3})$.  Although appealing due to its simplicity,
this scheme suffers some drawbacks. It is only appropriate on uniform finite
difference grids, and encounters difficulties discretizing $u$ near the boundary
of the domain.

Recent work by Froese \cite{froese_meshfree_2017}  treats the more general case
where $u$ is discretized on a cloud of point $\mathcal G$. Froese presents a
monotone finite difference scheme for the second derivative which is $\mathcal
O(R + d\theta)$. The parameter $R$ is a search radius, which will be defined
more precisely later.  Set $h = \sup_{ x\in \Omega} \min_{ x_j \in \mathcal G}
\norm{   x -  x_j }$.  Then  (as in the previous method) for a desired angular
resolution, $R$ is $\mathcal O\left( \frac{h}{d\theta}\right)$, and so with the
optimal choice of $d\theta = \sqrt{h}$, the method is formally $\mathcal O\left(
\sqrt{h} \right)$.  Unfortunately this scheme does not generalize easily to
higher dimensions.

In what follows, we present a monotone and consistent finite difference scheme
for the first and second derivatives which overcomes the deficiencies of the
preceding two methods. For the second derivative, if the grid is not symmetric,
our scheme has accuracy $\mathcal O(R + d\theta^2)$, or formally $\mathcal
O(h^{\frac{2}{3}})$. Further in the symmetric case, the scheme is $\mathcal
O(R^2 + d\theta^2)$, and is formally $\mathcal O(h)$. The method works in
dimension two and higher, and can be used on any set of discretization points,
uniform or otherwise.  It can easily be used near the boundary of a domains. In
particular, the scheme easily handles Neumann boundary conditions on non
rectangular domains.

Using these schemes as building blocks, we build monotone, stable
and consistent schemes for non linear degenerate elliptic equations on arbitrary
meshes.

Table \ref{tab:tab1} presents a summary of the second derivative schemes discussed in this paper.

\begin{table}
  \centering
  \begin{tabular}{p{3cm} | p{2cm} | p{1.5cm} | p{1.5cm} | p{4cm}}
    \hline
    Scheme & Order & Optimal $d\theta$ & Formal accuracy & Comments \\
    \hline
    \hline
    Nearest grid direction \cite{oberman_wide_2008} & $\mathcal O(r^2 +
      d\theta)$ &$\mathcal O(h^\frac{2}{3})$ &
    $\mathcal O(h^\frac{2}{3})$ &  Uniform grids. Difficulty at
    boundaries.\\
    Froese \cite{froese_meshfree_2017} & $\mathcal O (r + d\theta)$ & $\mathcal
    O (h^\frac{1}{2})$ & $\mathcal O (h^\frac{1}{2})$ &
    2d, mesh free. No problem at boundary. \\
    Linear interpolant, symmetric & $\mathcal O(r^2 + d\theta^2)$ & $\mathcal O
    (h^\frac{1}{2})$ & $\mathcal O(h)$ & $n$-d, uniform grids. No problem at boundary.  \\
    Linear interpolant, non symmetric & $\mathcal O(r + d\theta^2)$ & $\mathcal O
    (h^\frac{1}{3})$ & $\mathcal O(h^\frac{2}{3})$ & $n$-d, mesh free. No problem at boundary.  \\
    \hline
  \end{tabular}
  \caption{Comparison of the discretizations.}
  \label{tab:tab1}
\end{table}

\subsection{Directional discretizations}
The basic building block of our discretization are first and second order directional derivatives.
This is in contrast to the work of Mirebeau, where two dimensional shapes built up of triangles are chosen to match the ellipticity of the operator.

Write the first and second directional derivatives of a function $u$ in the direction $\w$
(with $||\w|| = 1$) as
\begin{align*}
  u_\w = \langle \w, D u \rangle,
  \qquad
  u_{\ww} = \w^\trp D^2 u \w.
\end{align*}
where $Du$ and $D^2 u$ are the gradient and Hessian of $u$, respectively.

Define the forward difference in the direction $v$ by
\[
  \mathcal D_{v}u( x) = \frac{u( x + v) -  u( x)}{|v|}
\]
The first order monotone finite difference schemes for $u_\w$ in the directions $t\w$ and $-t\w$ are given by
\begin{align}\label{eq:fd_1st}
  \mathcal D_{t\w}u( x) &= u_\w( x)  + \mathcal O(t) \\
  \nonumber \mathcal D_{-t\w}u( x)   &= u_\w( x) + \mathcal O(t)
\end{align}
The simplest finite difference scheme for $u_\ww$ is the centred finite differences
\begin{align}
  \frac{u( x + t \w) - 2u( x) +  u( x - t \w )}{t^2} \label{eq:fd22}
  &= \frac 1 t\left[  \mathcal D_{t\w}u( x) + \mathcal D_{-t\w}u( x)\right]
  \\&= u_{\ww}( x)  + \mathcal O(t^2) \label{eq:fd}
\end{align}
The generalization to unequally spaced points is clear from \eqref{eq:fd22}
\begin{align}
  \frac{2}{t_p + t_m} \left[
    \mathcal D_{t_p\w}u( x) + \mathcal D_{-t_m\w} u( x)
  \right] &= u_{\ww}( x) +  \mathcal O(t_+).
  \label{eq:fd_uneven}
\end{align}
where $t_+ =\max\{t_p, t_m\}$ (in general, the scheme is first order accurate, unless $t_p = t_m$).

\subsection{Directional finite differences using barycentric coordinates}
Suppose we want to compute $u_\w( x_0)$ using values $u(x_i)$ which
determine a simplex. Using linear interpolation, we can approximate the value of
$u( x+t_p\w)$ on the boundary of the simplex.
A convenient expression for this value is given by  using barycentric coordinates, (see, for example,~\cite[\S 5.4 p.595]{dahlquist2008numerical}), which allows us to generalize~\eqref{eq:fd22}.

Suppose  $\mathcal S_m$ and $\mathcal S_p$ are the vertices of an  $(n-1)$-dimensional simplex.
Suppose further that
\[
  r \leq \norm{ x_0 -  x_i} \leq R, \qquad \text{ for all }  x_i \in \left\{ \mathcal S_m, \mathcal S_p \right\}
\]
Suppose further that
\begin{align*}
  x_p &= 	x_0 + t_p w \text{ is in the simplex determined by } \mathcal S_p
  \\
  x_m &= x_0 - t_m w \text{ is in the simplex determined by } \mathcal S_n
\end{align*}
for $t_m, t_p \in [r,R]$.
Construct the corresponding linear interpolants $L_m$ and $L_p$
\begin{align}
  L_p( x) &= \sum_{i \in \mathcal S_p} \lambda_p^i( x) u( x_i)
  \label{eq:interpp}
  \\  L_m( x) &= \sum_{i \in \mathcal S_m} \lambda_m^i( x) u( x_i). \label{eq:interpm}
\end{align}

Here $\lambda_{p}( x)$ and $\lambda_{m}( x)$ are the barycentric coordinates in
$\mathcal S_p$ and $\mathcal S_m$ respectively.  The barycentric coordinates are easily
constructed. Let $ v^{p}_i =  x_i -  x_0$, $i \in \mathcal S_p$, and
similarly define $ v^m_i$.
By assumption all $ v_i$'s satisfy $r\leq || v_i|| \leq R$.  Let
$V_p$ be the matrix
\begin{equation}
  V_p = \begin{bmatrix}
     v^{p}_1 &   v^{p}_2&  \dots &  v^{p}_n \\
  \end{bmatrix}.
  \label{eq:T}
\end{equation}
Then $ \lambda_{p}$ is given by solving
\begin{equation}
  V_p  \lambda_{p} =  x.
  \label{eq:lambda}
\end{equation}
The barycentric coordinates $\lambda_m$ for $\mathcal S_m$ are defined analogously.
By virtue of convexity, if $ x$ lies in the (relative) interior of a simplex, its
barycentric coordinates are positive and sum to one.

Barycentric coordinates allow us to define the finite difference schemes  for the first and second directional derivatives as follows.
\begin{definition}[First derivative schemes]
  The first derivative scheme takes two forms, respectively upwind and downwind:
  \begin{align}
    \mathcal D_\w u( x_0) &:= \frac{1}{t_p}\left(
    L_p( x_0 + t_p  \w) - u( x_0) \right), \qquad t_p = \frac{1}{ 1^\trp V_p^{-1}\w}
    \label{eq:scheme_1st}
    \\
    \mathcal D_{-\w} u( x_0) &:= \frac{1}{t_m}\left(
    L_p( x_0 - t_m  \w) - u( x_0) \right), \qquad t_m = \frac{-1}{ 1^\trp V_m^{-1}\w}
    \label{eq:scheme_1st_down}
  \end{align}
  \label{def:1st}
\end{definition}
\begin{definition}[Second derivative scheme]
The second derivative scheme is defined as
  \begin{equation}
    \mathcal D_{\ww} u( x_0)= \frac{2\left(\mathcal D_\w u( x_0) + \mathcal
    D_{-\w}u( x_0)\right)}{t_p+t_m}.
    \label{eq:scheme_uneven}
  \end{equation}
  with $t_p$ and $t_m$  given above.
  \label{def:2nd}
\end{definition}

\begin{lemma}[Monotone and stable]
  The finite difference schemes of Definitions \ref{def:1st} and \ref{def:2nd}
  are monotone and stable.
  \label{lem:ms}
\end{lemma}
\begin{proof}
  By convexity, we are guaranteed that $0\leq \lambda_{p,m}^i \leq
  1$. Further, we have that both $\sum \lambda_p^i = \sum \lambda_m^i = 1$. This corresponds to a monotone discretization of the operator \cite{ObermanDiffSchemes}.
\end{proof}

In the application below, we will use long, slender simplices, which are
oriented near the directions $\pm w$, and control the interior and exterior radii, in order to establish the accuracy of the schemes.

\section{The framework}\label{sec:framework}
In this section we introduce a framework for constructing monotone finite difference
operators on a point cloud, in dimensions two or three.  To implement the method, we require finding triangles (in two dimensions) or simplices (in three dimensions) which contain the reference point.   The configuration of these simplices determines the accuracy of the scheme.

\subsection{Notation}
We use the following notation.
\begin{itemize}
  \item $\Omega \subset \R^n$, an open convex bounded domain with Lipshitz boundary
    $\partial \Omega$. We focus on the cases $n=2$ and $n=3$.
  \item $\mathcal G\subset \bar \Omega$ is a point cloud with points $ x_i,
    i=1\dots N$.
  \item If $\mathcal G$ is given as the undirected graph of a
    triangulation, then $A$ is the corresponding adjacency matrix of the graph.
  \item $h = \sup_{ x \in \Omega} \min_{ y \in \mathcal
    G}\norm{ x -  y}$, the spatial resolution of the graph. Every ball
    of radius $h$ in $\bar \Omega$ contains at least one grid point.
  \item $h_B = \sup_{ x \in \partial\Omega} \min_{ y \in \mathcal G \cap
    \partial\Omega} \norm{ x -  y}$ is the spatial resolution of the
    graph on the boundary.
  \item $\delta = \min_{ x \in \mathcal G\cap\Omega  } \min_{ y \in
    \mathcal G \cap \partial\Omega} \norm{ x -  y}$ is minimum
    distance between an interior point and a boundary point.
  \item $\ell$ is the minmum length of all edges in the graph $\mathcal G$.
  \item $d\theta$ is the desired angular resolution. We shall require at least
    $d\theta<\pi$.
  \item $R = C_n h \left( 1 + \cosec(\frac{d\theta}{2}) \right)$ is the maximal search radius, and depends
    only on the angular resolution, the spatial resolution, and a constant $C_n$
    determined by the dimension.
  \item $r = C_n h \left( -1 + \cosec (\frac{d\theta}{2})\right)$ is the minimal
    search radius. We will see that the minimal search radius is necessary to
    guarantee convergence of the schemes. Further, to guarantee the
    convergence of schemes near the boundary, it will be necessary to require
    $\delta \geq r$.
  \item $C_n$ is a constant determined by the dimension. In $\R^2$, $C_2 = 2$;
    in $\R^3$, $C_3 = 1+\frac{2}{\sqrt{3}}$.
\end{itemize}

The construction of the schemes above require the existence of simplices
which intersect the vector $\w$. For accuracy, we further require that angular resolution
of the simplices diameter relative to the point $ x_0$ is less than $d\theta$.
The following three lemmas show that for given angular and spatial resolutions,
such schemes exist.  Refer to Figure~\ref{fig:simplex_exist}.

\begin{figure}[t]
  \centering
  \begin{subfigure}[b]{0.48\textwidth}
    \trimbox{-1cm -.5cm -1cm 0cm}{\begin{tikzpicture}[scale=2]
  \draw (1,1) circle (1);

  \draw[thick,->,>=latex] (0,0) -- (2,2) node [above
  right] {$ w$};

  \draw[thick,<->,color=cbblue] (1.3536,.6464) -- (1.7866,.8964) node [pos=0.5,
  below right,inner sep=0.1] {$h$};

  \draw[thick,<->,color=cbred,inner sep=0.1] (1,1) -- (1.866,1.5) node [pos=0.6,
  below right] {$C_2h$};

  \draw (1.3536,.6464) circle (.5);
  \draw (.6464,1.3536) circle (.5);

  \draw[dashed] (0,2) -- (2,0) node [below right] {$ w^\perp$};

  \draw[fill=cbpurple,cbpurple] (1.25, .25) circle (0.03) node [above right] {$ x_i$};
  \draw[fill=cbpurple,cbpurple] (.35, 1.35) circle (0.03) node [above ] {$ x_j$};
  \draw[cbpurple] (1.25,.25) -- (.35,1.35) node [above right,pos=0.7,inner sep=0.2] {$\mathcal S$};

  \draw[fill=black] (1.3536,.6464) circle (0.01) ;
  \draw[fill=black] (.6464,1.3536) circle (0.01) ;
  \draw[fill=black] (1,1) circle (0.01) ;
\end{tikzpicture}}
    \caption{$C_2 = 2$}
    \label{fig:2dball}
  \end{subfigure}
  \hfill
  \begin{subfigure}[b]{0.48\textwidth}
    \trimbox{-1cm -0.5cm -1cm 0cm}{\begin{tikzpicture}[scale=1.15]
  \draw[fill=cbpurple, cbpurple] (1.3,-0.4) circle (0.045) node [below left] {$P  x_i$};
  \draw[fill=cbpurple, cbpurple] (-0.5,-0.6) circle (0.045) node [left] {$P  x_j$};
  \draw[fill=cbpurple, cbpurple] (-.3,1.6) circle (0.045) node [below left] {$P  x_k$};
  \draw[fill=cbpurple, cbpurple, fill opacity=0.2] (1.3,-0.4) -- (-0.5,-0.6) --
  (-0.3,1.6) node [below, opacity=1,pos=0.5,right, inner sep=5] {$\mathcal S$}-- (1.3,-0.4);

  \draw (1,0) circle (0.866);
  \draw (-0.5,0.866) circle (0.866);
  \draw (-0.5,-0.866) circle (0.866);

  \draw (0,0) circle (1.866);

  \draw[thick,<->,color=cbblue] (1,0) -- (1.75,.433) node [pos=0.5,
  below right,inner sep=0.1] {$h$};
  \draw[thick,<->,color=cbred] (-0.045,0) -- (-1.866,0) node [pos=0.9,
  above right,inner sep=0.5] {$C_3 h$};

  \draw[fill=black] (0,0) circle (0.045) node [right] {$P  w$};

  \draw[fill=black] (1,0) circle (0.016);
  \draw[fill=black] (-0.5,0.866) circle (0.016);
  \draw[fill=black] (-0.5,-0.866) circle (0.016);

\end{tikzpicture}}
    \caption{$C_3 = 1+\frac{2}{\sqrt{3}}$}
    \label{fig:3dball}
  \end{subfigure}
  \caption{There exists an $n-1$ simplex $\mathcal S$ enclosing $\w$,
    contained within ball of radius $C_n h$.
    In Fig \ref{fig:3dball}, projections onto a plane perpendicular to
  $\w$ are shown.}
  \label{fig:simplex_exist}
\end{figure}
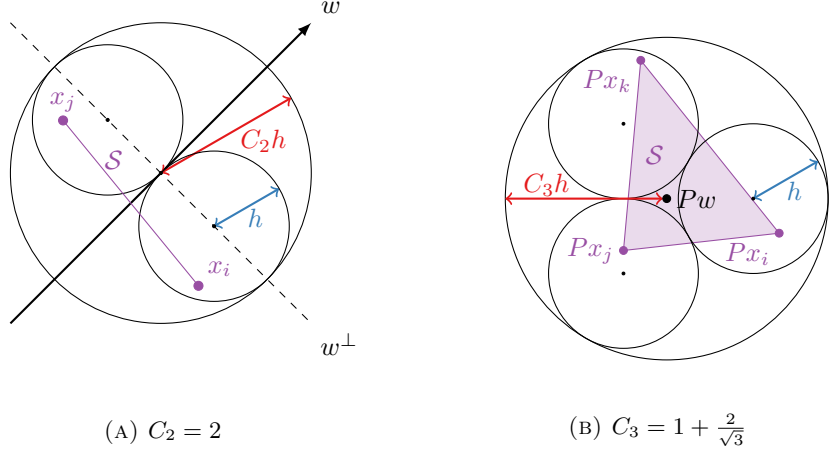

\begin{lemma}[Existence of scheme away from boundary]
  Take $ x_0 \in \mathcal G$ with $\dist( x_0,\partial\Omega) \geq R$.
  Then it is possible to construct the simplices used in Definition~\ref{def:1st}.
  \label{lem:ex_int}
\end{lemma}
\begin{proof}
  We must show that $\mathcal S_p$ and $\mathcal S_m$ exist. We first show the
  existence of the simplex $\mathcal S_p$; $\mathcal S_m$ follows similarly.
  Define the cone
  \begin{equation}
    K:= \left\{  x \mid   \frac{\langle   v,
      w\rangle}{\norm{ v}} \geq 1 - \cos(\frac{d\theta}{2}),  v =
      x -  x_0
    \right\}.
    \label{eq:K}
  \end{equation}
  Any two points in $K$ have angular resolution (relative to $ x_0$) less than
  $d\theta$. Therefore choosing points in this cone ensures the angular
  resolution is satisfied.

  We must now show that the set $\mathcal G\cap K$ contains points defining
  $\mathcal S_p$.  By construction, any ball in the interior of $\Omega$ with
  radius $h$ contains at least one interior point. Therefore, we may construct a
  simplex intersecting the line $ x_0 + t \w$, $t \in \R$, by placing $n$
  kissing balls on a plane $\w^\perp$ perpendicular to $\w$, and choosing a
  point from within each ball. Using simple geometrical arguments (cf
  Apollonius' problem), it can be shown that these $n$ balls of radius $h$ are
  all contained within a larger ball of radius $C_n h$ (with $C_2 = 2$ and $C_3
  = 1+\frac{2}{\sqrt{3}}$). Refer to Figure \ref{fig:simplex_exist}. Thus, a
  candidate simplex is guaranteed to exist within every ball of radius $C_n h$
  with center on the line $ x_0 + t \w$.

  Let this larger ball be $\bar B( x_0 + (R-C_n h) \w, C_n h)$. See
  Figure \ref{fig:interior_search}.
  Simple trigonometric arguments show that this ball is contained within the
  cone $K$. Therefore the cone $K$ contains the desired simplex
  $\mathcal S_p$.

  Similar reasoning gives the existence of $\mathcal S_m$. Taken together, this
  allows for the construction of the schemes.
\end{proof}

\begin{figure}[t]
  \centering
  \begin{subfigure}[b]{0.7\textwidth}
    \trimbox{-1cm -.5cm -1cm 0cm}{\begin{tikzpicture}[scale=2]
  \draw (1.5,1.5) circle (1);

  \draw[thick,<->,inner sep=0.1,cbred] (1.5,1.5) -- (2.21,2.21) node [pos=0.5,
  below right] {$C_nh$};
  \draw[thick,->,>=latex] (2.21,2.21) -- +(.5,.5) node [above right] {$ w$};

  \draw[cbpurple] (2.05,1) -- (.85,1.85) node [below left,pos=0.2,inner sep=0.2] {$\mathcal S$};
  \draw[fill=cbpurple,cbpurple] (2.05,1) circle (0.03) node [above right] {$ x_i$};
  \draw[fill=cbpurple,cbpurple] (.85, 1.85) circle (0.03) node [above right] {$ x_j$};

  \draw[fill=black] (0,0) circle (0.03) node [below right] {$ x_0$};

  \draw[dashed] (0,0) -- (16.8745:3.1213) node [below right,pos=0.5] {$R$};
  \draw[dashed] (0,0) -- (73.1255:3.1213) ;
  \draw[dashed] (16.8745:3.1213) arc (16.8745:73.1255:3.1213) node [above right]
  {$K\cap \bar B( x_0,R)$};
  \draw (0.5431,1.7903) -- (1.5,1.5) -- ++(225:0.4714)  -- (0.5431,1.7903) node
  [below,pos=0.6] {$b$};

  \draw[thick,<->,cbblue] (1.5,1.5) -- +(225:0.4714) node [pos=0.5,below right]
  {$a$} ;
  \draw[thick,<->,cbgreen] (0.016,0.016) -- (45:1.6499) node [pos=0.5,below right]
  {$c$};

  \draw (0.2,0.2) arc (45:73.1255:0.2828) node [above right,pos=0.85]
  {$\frac{d\theta}{2}$};

  \draw (0.5431+0.2,1.7903-0.2) arc (-45:-16.8745:0.2828) node [below right,pos=0.85]
  {$\frac{d\theta}{2}$};

  \node[align=left] at (3.75,2.15) {$\openup -1 \jot
    \begin{aligned}a &= C_n h \sin(d\theta/2)\\
      b &= C_n h \cos(d\theta/2)\\
      c &=b \cot(d\theta/2) \\
      R &= C_n h + a + c\\
      &= C_n h( 1+ \mathrm{cosec}(d\theta/2))
    \end{aligned}$};
\end{tikzpicture}}
    \caption{Interior simplex}
    \label{fig:interior_search}
  \end{subfigure}
  \hfill
  \begin{subfigure}[b]{0.28\textwidth}
    \trimbox{-1cm -.5cm -1cm 0cm}{\begin{tikzpicture}[scale=2]
  \draw[fill=black] (0,0) circle (0.03) node [below right] {$ x_0$};
  \draw (45:2.5) arc (45:75:2.5) node [above right] {$\partial \Omega$}; 
  \draw[cbpurple, fill=cbpurple] (50:2.5) circle (0.03) node [below left] {$ x_i$};
  \draw[cbpurple, fill=cbpurple] (66:2.5) circle (0.03) node [below left] {$ x_k$};
  \draw[cbpurple] (50:2.5) -- (66:2.5) node [pos=0.4, below left,inner sep=0.2] {$\mathcal S$};
  \draw (0,0) -- (45:1.8);
  \draw (75:1.8) -- (0,0);
  \draw[thick,<->] (75:1.8) -- (45:1.8) node [below right,inner sep=0.15] {$2\delta \tan
    (\frac{d\theta}{2})$};
    \draw[thick,<->] (60:0.04) -- (60:1.74) node [above left,pos=0.6] {$\delta$};
  \draw[dashed] (75:1.8) -- (75:2.5);
  \draw[dashed] (45:1.8) -- (45:2.5);
  \draw[->, >=latex] (60:1.74) -- (60:3) node [above right] {$ w$};
  \draw (60:0.3) arc (60:75:0.3) node [left] {$\frac{d\theta}{2}$};

\end{tikzpicture}}
    \caption{Boundary simplex}
    \label{fig:boundary_search}
  \end{subfigure}

\end{figure}
\begin{lemma}[Existence of interior scheme near boundary]
  Take $ x_0 \in \mathcal G \cap \Omega$ with $ \dist( x_0,\partial\Omega) < R $.
  If the spatial resolution of $\mathcal G$ on the boundary is such that
  $C_n h_B \leq \delta \tan(\frac{d\theta}{2})$ and the angular resolution is
  small enough (dependent on the regularity of the boundary) then the schemes
  given by Definition~\ref{def:1st} exists.
  \label{lem:ex_bdry}
\end{lemma}
\begin{proof}
  We first will show $S_p$ exists; the existence of $\mathcal S_m$ follows
  analogously. With the cone $K$ defined as in the previous lemma, we must show that
  $\mathcal G \cap K$ contains points defining $\mathcal S_p$.

  Suppose first that $\bar B( x_0, R) \cap K \subset \Omega$. Then the
  existence of $\mathcal S_p$ follows from Lemma \ref{lem:ex_int}.

  Suppose instead that $\bar B( x_0, R) \cap K$ is not entirely contained
  within $\Omega$. If $d\theta$ is small enough, then the boundary is contained
  within $\bar B( x_0, R) \cap K$,
  \begin{align}
    \norm{ x_0 -  y} < R\text{ if $ y\in \partial\Omega\cap K$}.
  \end{align}
  By construction, $\dist( x_0, \partial\Omega) \geq \delta$. Therefore the
  diameter of this portion of the boundary is at least $\delta \tan
  (\frac{d\theta}{2}) \geq  C_n h$. Using similar geometrical reasoning as in the
  previous lemma (see Figure \ref{fig:boundary_search}), there must be $n$
  points on the boundary defining the simplex $\mathcal S_p$.
\end{proof}

The previous two lemmas guarantee the first and second derivative schemes exist
on the interior of the domain. The existence of the first derivative scheme on
the boundary is, in general, not a simple exercise: existence depends on the
regularity of the domain, the angle formed by $\w$ and the boundary normal $
n$, $h$, $h_B$, and $\delta$. For our purposes, we guarantee the existence of a
scheme for the normal derivative with the following lemma.
\begin{lemma}[Existence of normal derivative scheme on the boundary]
  Define the set $\Omega_\delta:=\left\{  x \in \Omega \mid \dist( x,\partial
  \Omega) \geq \delta \right\}$. Suppose $\Omega_\delta$ is such that for every
  $ x\in\Omega_\delta$, $ x \in \bar B(\y,C_d h) \subset \Omega_\delta$ for some
  $\y \in \Omega_\delta$. Suppose
  further that minimum distance between interior points and boundary points is
  less than the minimum search radius, $\delta\leq r$. Then the scheme
  $\mathcal D_{ n} u( x_0)$ for the inward pointing normal derivative
  exists for all boundary points.
  \label{lem:ex_normal}
\end{lemma}
\begin{proof}
  Let $ x_0$ be a boundary point. If $\delta< r$ then the search ball $\bar
  B( x_0 + (R-C_n h) n, R)$ is contained entirely within $\Omega$. Thus, by the
  same arguments as in the proof of Lemma \ref{lem:ex_int}, the simplex $\mathcal
  S_p$ exists and has angular resolution less than $d\theta$. This allows for the construction of $\eqref{eq:scheme_1st}$ for
  the normal derivative.
\end{proof}

Combining these three lemmas guarantees existence of the schemes.
\begin{theorem}[Existence of schemes]
  Suppose $\mathcal G$ is a point cloud in $\Omega$ with boundary resolution $C_b h \leq
  \delta \tan (\frac{d\theta}{2})$. With small enough $d\theta$, the
  first and second derivative schemes defined respectively by
  Definitions \ref{def:1st} and \ref{def:2nd} exist for all interior
  points $ x_0 \in \mathcal
  G \cap \Omega$. If in addition every $ x \in \Omega_\delta$ lies within a ball $\bar
  B_{C_d h} \subset \Omega_\delta$ and $\delta < r$, the scheme $\mathcal
  D_{ n} u(x_0)$ for the inward normal
  derivative exists for all boundary points $ x_0 \in\mathcal G\cap
  \partial\Omega$.
\end{theorem}

\subsection{Consistency \& Accuracy}

We now derive bounds on the error of the schemes, and show that the schemes are
consistent with an appropriate choice of $d\theta$ in terms of $h$.  First,
recall the fact that the first term for the error of a linear interpolant is
given by
\begin{equation}
  u( x) - L( x) \approx \frac{1}{2}\sum \lambda_j( x) ( x -
  x_j)^\trp D^2 u( x_j) ( x -  x_j).
  \label{eq:err1}
\end{equation}
Therefore the interpolation error at $ x_0 + t_p \w$ is
\begin{align}
  E[L_p] :=& u( x_0 + t_p\w) - L_p( x_0 + t_p\w) \\
  \approx & \frac{1}{2} \sum_{i \in \mathcal S_p} \lambda^i_p \left(  v^p_i
    - t_p
  \w \right)^\trp D^2u( x_i) \left(  v^p_i - t_p \w \right) \\
  \leq & \frac{1}{2}||D^2 u||_\infty \sum_{i \in \mathcal S_p} \lambda^i_p||
   v^p_i - t_p
  \w||^2.
  \label{eq:err2}
\end{align}
The interpolation error at $ x_0 - t_m \w$ is bounded above in a similar
fashion.

\begin{lemma}[Consistency of first derivative scheme]
  The first derivative schemes of Definition \ref{def:1st} are consistent with a formal
  discretization error of $\mathcal O(h)$.
  \label{lem:1st_cons}
\end{lemma}

\begin{proof}
  The angular resolution error of the upwind first derivative scheme is
  \begin{align}
    E[\mathcal D_\w u,d\theta] &=  \frac{E[L_p]}{t_p}\\
    &\leq \frac{1}{2}\norm{D^2u}_\infty \sum_{i\in\mathcal S_p}  \lambda_i^p
    \frac{\norm{ v_i^p - t_p \w}^2}{t_p} \\
    &\leq \frac{1}{2}\norm{D^2u}_\infty
    \frac{\max_{i,j\in\mathcal S_p}\norm{ v_i^p  -  v_j^p}^2}{\min_{k\in\mathcal
    S_p} \norm{ v_k}}.\label{eq:ang_1st}
  \end{align}
  By construction, the maximum distance between any two points in a simplex of the
  scheme is $2 C_n h$, and so the numerator here is bounded above by $(2 C_n
  h)^2$. Further, the minimum distance of a vector in the scheme is bounded below
  by the minimum search radius $r$. That is
  \begin{align}
    \min_{k \in \mathcal S_p, \mathcal S_m} \norm{ v_k}
    \geq r &= C_n h \left( -1 + \cosec (\frac{d\theta}{2}) \right)\\
    &=\mathcal O\left(\frac{h}{d\theta}\right).
  \end{align}
  With this in mind, \eqref{eq:ang_1st} is
  bounded by
  \begin{align}
    E[\mathcal D_\w u,d\theta] &\leq \frac{1}{2}\norm{D^2u}_\infty \frac{(2 C_n
    h)^2}{r} \label{eq:acc_1st}\\
    &=\mathcal O(h d\theta)
  \end{align}
  Fixing $d\theta$ constant as $h \rightarrow 0$ gives that the scheme is
  $\mathcal O(h)$.
\end{proof}

\begin{lemma}[Consistency of second derivative schemes]
  Using a non symmetric stencil, with the optimal choice $d\theta =
  \left(\frac{h}{2}\right)^{\frac{1}{3}}$, the second derivative scheme
  $\mathcal D_{\ww} u$ of Definition \ref{def:2nd} is
  consistent, with a formal accuracy of $\mathcal O(h^{\frac{2}{3}})$.
  Moreover on a symmetric stencil, with the optimal choice $d\theta =
  h^{\frac{1}{2}}$, $\mathcal D_{\ww} u$ is consistent, with a formal
  accuracy of $\mathcal O(h)$.
  \label{lem:2nd_cons}
\end{lemma}
\begin{proof}

  The angular resolution error of the second derivative scheme is
  \begin{align}
    E[\mathcal D_{\ww} u,d\theta] &= 2\left( \frac{E[L_p]}{t_p^2 + t_p t_m} +
    \frac{E[L_m]}{t_m^2 + t_p t_m} \right ) \label{eq:err_ang1} \\
    &\leq \frac{1}{t_-^2} \Big(E[L_p] + E[L_m]\Big)
  \end{align}
  where $t_- =\min\{t_p, t_m\}$. Arguing in a similar fashion as in the first
  derivative,
  \begin{align}
    E[\mathcal D_{\ww} u,d\theta] &\leq ||D^2 u||_\infty \frac{\max_{S \in \mathcal S_p, \mathcal
    S_m} \max_{i,j \in S} || v_i -  v_j||^2}{ \min_{k \in \mathcal S_p,
    \mathcal S_m} || v_k||^2}
    \label{eq:ang_err}\\
    &\leq ||D^2 u||_\infty \frac{(2 C_n h)^2}{r^2} \\
    & =\mathcal O(d\theta^2).
    \label{eq:ang_err2}
  \end{align}
  since $d\theta = \mathcal O(\frac{h}{r})$ when $d\theta$ is small.

  The total error of the scheme is the sum of angular and spatial resolution
  errors. For the second derivative, in the non symmetric case, the error of the scheme is
  \begin{align}
    E[u_{\ww}] &= \mathcal O(R + d\theta^2)\\
    &= \mathcal O(\frac{h}{d\theta}+ d\theta^2),
  \end{align}
  because $R = \mathcal O(\frac{h}{d\theta})$ when $d\theta$ is small.
  In the symmetric case the error is
  \begin{align}
    E[\mathcal D_{\ww} u] &= \mathcal O(R^2 + d\theta^2) \\
    &= \mathcal O\left( \left(\frac{h}{d\theta} \right)^2 + d\theta^2\right)
  \end{align}
  To ensure the scheme is consistent, $d\theta$ must be chosen in terms of $h$
  such that the error of the scheme goes to zero as the point cloud is refined.
  In the non symmetric case, the best choice  is $d\theta =
  \left(\frac{h}{2}\right)^\frac{1}{3}$, which gives a formal accuracy of
  $\mathcal O(h^\frac{2}{3})$. When the discretization is symmetric, the best
  choice of $d\theta$ is $\sqrt{h}$, and the scheme is formally $\mathcal O(h)$.
\end{proof}

\begin{remark}
  To guarantee the accuracy of the first order scheme, $d\theta$ must remain
  constant as $h \rightarrow 0$. In contrast, for the second order scheme to
  converge as $h \rightarrow 0$, it must be that $d\theta \sim
  \left(\frac{h}{2} \right)^{\frac{1}{3}}$ (on a non symmetric grid). Thus, for the remainder of the
  paper, when we speak of the angular resolution error, we mean the angular
  resolution error for the second derivative scheme. We assume that the angular
  resolution error for the first derivative scheme has been fixed to some
  reasonable constant, say $\frac{\pi}{4}$.
\end{remark}

\begin{remark}
  To ensure the existence of consistent schemes near the boundary, we require
  that the minimal distance between interior and boundary points is greater than
  the minimal search radius, $\delta \geq r$.
\end{remark}

\subsection{Practical considerations}
We now outline a procedure for preprocessing the point cloud $\mathcal G$, which will greatly
speed the construction of elliptic schemes. The algorithm takes a point cloud
$ x_i\in\mathcal G, i\in\mathcal I$ and returns a set $\mathcal L_i$ of candidate simplices for each
point. Each simplex $\mathcal S_k \in \mathcal L_i$, $k=1,\dots,m_i$, is contained within the
annulus formed by the minimum and maximum search radii. Further, projecting
$\mathcal L_i$ onto the sphere forms a covering of the sphere. Thus all possible
directions are available.

The pseudocode of the algorithm is given in Algorithm \ref{alg:pre}.
Note that we assume the set of normalized neighbour points, denoted by $V$, is
unique. If not, for each set of non unique points, keep only the smallest in
norm.

Now suppose the list of simplices \[\mathcal L_i = \left\{ S_k \right\},
\quad k=1,\dots,m_i\]
have been generated for a point $ x_i$. Given a direction $\w$ it
is straight forward to choose $\mathcal S_p$ and $\mathcal S_m$ from $\mathcal
L_i$.  Define
\[V_k = \begin{bmatrix}
     v_1 &   v_2&  \dots &  v_n \\
\end{bmatrix},\quad\text{ with }  v_k =  x_j -  x_i,\, j \in \mathcal S_k.\]
Then by Farkas' lemma,
\[
  \mathcal S_p = \left\{\mathcal S_k\in \mathcal L_i \mid
  V_k^{-1} \w \geq 0\right\}
\]
and
\[
  \mathcal S_m = \left\{\mathcal S_k\in \mathcal L_i \mid
  V_k^{-1} \w \leq 0\right\}.
\]
If these sets are not singletons (when $\w$ aligns with a grid direction), then
choose one representative element.

\begin{algorithm}[t]
  \caption{Algorithm for preprocessing the point cloud}
  \label{alg:pre}
  \SetKwInOut{Input}{Input}
  \SetKwInOut{Output}{Output}

  \Input{A point cloud $ x_i\in\mathcal G$ in $\R^n$, $i\in \mathcal I$, and
  resolution error $d\theta$}
  \Output{A list of sets of simplices $\mathcal L_i$,
    $i\in\mathcal I$, where $\mathcal L_i=\left\{\mathcal S_1, \dots, \mathcal
  S_{m_i}  \right\}$}

  $\mathcal T \gets \text{triangulation}(\mathcal G)$ \tcp*{triangulation of
  $\mathcal G$}

  $A \gets \text{adj}(\mathcal T)$ \tcp*{Adjacency matrix of $\mathcal T$}

  $\ell \gets \text{ minimum length of all edges in } \mathcal T$ \;

  $h \gets \sup_{ x \in \Omega} \min_{\y \in \mathcal G} \norm{ x-\y}$
  \tcp*{spatial resolution of point cloud}

  $R \gets C_n h \left(1 + \cosec\left(
  \frac{d\theta}{2}\right) \right)
  $ \tcp*{maximum search radius}

  $r \gets C_n h \left( -1 + \cosec\left(
  \frac{d\theta}{2}\right) \right)
  $ \tcp*{ minimum search radius}

  $p\gets \left\lceil\frac{R}{\ell}\right\rceil$ \tcp*{maximum neighbour
  graph distance}

  $P\gets\sum_{k=1}^p A^k$ \;

  \ForEach{$i \in \mathcal  I$}{
    $\mathcal N \gets \left\{j \mid P_{ij}\neq 0,\, i\neq j,\, r \leq \norm{
      x_i- x_j}
    \leq R  \right\}$ \tcp*{Neighbour indices}

    $V \gets \left\{ \frac{ x_i -  x_j}{\norm{ x_i -  x_j}} \mid j \in \mathcal
    N \right\}$ \tcp*{assume elements of $V$ are unique}

    $C \gets \text{Convex hull of }V$ \;
    $\mathcal L_i \gets \emptyset$ \;
    \ForEach{Facet $\mathcal F$  of $C$}{
      \tcp{$\mathcal F$ is a set of indices of the points in $V$ }
      $\mathcal S \gets \left\{  x_k \mid k=\mathcal N_j,\, j\in \mathcal F
      \right\}$ \;

    $\mathcal L_i = \mathcal L_i \cup \{\mathcal S\}$ \; }

  }
  \Return{$\left\{ \mathcal L_i \right\},\, i\in \mathcal I$}
\end{algorithm}

\begin{remark}
  The  proofs of Section \ref{sec:framework}  relied on choosing the maximal and minimal search radii to
  respectively be $R,r = C_n h(\pm 1 + \cosec(\frac{d\theta}{2}))$. This choice makes
  the proofs relatively straightforward. However, it is possible to still
  guarantee existence and accuracy of the finite difference scheme with the
  narrower band of search radii $R, r = h(\pm 1 +  C_n \cosec(\frac{d\theta}{2})$.
    In practice this set of search radii limits the appearance of `spikey'
    stencils. We have found that it is best to choose a set set of simplices
    whose boundary has minimal surface area, thus limiting the amount of
    interpolation error.
  \end{remark}

  \section{Application: Eigenvalues of the Hessian}\label{sec:eigs}
  It is relatively straight forward to employ $\mathcal D_{\ww} u$ to find
  maximal and minimal eigenvalues of the Hessian about a point $ x_i\in
  \mathcal G$.  We will illustrate the procedure for the maximal eigenvalue,
  but the procedure is analogous for the minimal eigenvalue.

  Define the finite difference operator $\Lambda_+^{h,d\theta}
  u( x_i):=\sup_{\norm{\w}=1} \mathcal D_{\ww} u( x_i)$ as the approximation
  of the maximum eigenvalue of the Hessian.

  Actually computing $\Lambda_+^{h,d\theta}u( x_i)$ reduces to an optimization
  problem. Define $K(\mathcal S)$ as the cone generated by a set $S$. We say that
  two cones overlap if their intersection is non empty. For each
  pair $\{\mathcal S_p, \mathcal S_m\}$ of overlapping antipodal simplices in
  $\mathcal L_i$ (with $K(\mathcal S_p) \cap K(-\mathcal S_m)\neq \emptyset)$, one
  computes
  \begin{equation*}
    \begin{aligned}
      P[\mathcal S_m, \mathcal S_p] =  &\underset{\lambda_p, \lambda_m}{\text{maximize}}
      & & 2\left[ \frac{\sum_{i \in \mathcal S_p} \lambda^i_p u( x_i) - u(
        x_0)}{t^2_p + t_p t_m} + \frac{\sum_{i \in \mathcal S_m} \lambda^i_m
          u( x_i) - u(
      x_0)}{t^2_m + t_p t_m}\right ] \\
      & \text{subject to} & & 0 \leq  \lambda_p,  \lambda_m \leq  1 \\
      & & & 1^\trp \lambda_p = 1 \\
      & & & 1^\trp \lambda_m = 1 \\
      & & & t_p = ||V_p  \lambda_p || \\
      & & & t_m = ||V_m  \lambda_m ||
    \end{aligned}
  \end{equation*}
  The variables $t_p$ and $t_m$ are dummy variables.  On a two dimensional uniform
  grid, this simplifies to a straightforward optimization problem over one
  variable, which can be solved analytically.

  To find the maximal eigenvalue, one takes the maximal value computed over all
  antipodal pairs:
  \begin{align}
    \Lambda_+^{h,d\theta}u( x_i) = \max_{\substack{\mathcal S_m,\, \mathcal S_p \in
    \mathcal L_i\\K(\mathcal S_m) \cap K(- \mathcal S_p) \neq \emptyset}} P[\mathcal S_m,
    \mathcal S_p] \label{eq:eigopt}
  \end{align}
  The error of the scheme is
  \begin{align}
    E[\Lambda_+^{h,d\theta}] &= \left |\max_{\norm{ v}=1}  v^\trp D^2
    u( x_i)  v -
    \max_{\norm{\w}=1} \mathcal D_{\ww} u( x_i) \right | \\
    &\leq \max_{\norm{\w}=1} w^\trp D^2 u( x_i)  w -
    \mathcal D_{\ww} u( x_i)\\
    &=\mathcal O(R + d\theta^2),
  \end{align}
  on a non symmetric grid. As before, on a symmetric grid the error is $\mathcal
  O(R^2+d\theta^2)$.

  \begin{remark}
    In cases other than on a symmetric grid in two dimensions, the
    optimization problem
    \eqref{eq:eigopt} is difficult to implement. As a compromise, one may
    instead compute finitely many directional derivative $\mathcal D_{ w_i
    w_i} u$, $i=1,\dots,k$. Define the \emph{effective} angular resolution through
    \begin{align}
      \cos d\theta_e = \max_i  \min_{j\neq i} \langle \w_i, \w_j \rangle.
    \end{align}
    Because the directional derivative may be taken off grid, one may choose
    sufficiently many
    directions $\left\{\w_i \right\}$ such that $d\theta_e \leq d\theta^2$. With
    this choice of directional derivatives, the maximal eigenvalue of the Hessian
    can be defined as
    \begin{equation}
      \Lambda_+^{h,d\theta_e} u( x_i) = \max_i \mathcal D_{\w_i \w_i} u(
      x_i).
    \end{equation}
    A simple computation shows that $\Lambda_+^{h,d\theta_e}$ also has accuracy
    $\mathcal O(R + d\theta^2)$.

  \end{remark}

  \section{Solvers}
  Before continuing with specific numerical examples, we first detail the
  numerical solver used. All solutions in Section \ref{sec:ex} were computed with
  a global semi-smooth Newton method. Without modification, the Newton method
  fails, because the Newton method is guaranteed to be only a local method. However,
  the Newton method is achieves supralinear rates of convergence when the starting
  condition is close enough to the true solution.

  Thus to guarantee convergence, we use a global semi-smooth Newton method
  \cite[Chapter~8]{ssnm}. Let $F^h[u]$ be a finite difference approximation of an
  elliptic operator $F[u]$. After each Newton step, we check for a sufficient
  decrease in the energy $\norm{F^h[u]}^2$. If the Newton step does not decrease,
  the method switches to performing Euler steps, which is a guaranteed descent
  direction. We perform Euler steps for the same amount of CPU time
  as one Newton step, which was first proposed in \cite{carrington}.
  Because the Euler step is a guaranteed descent direction, the method is globally
  convergent \cite{ssnm}.

  \section{Numerical Examples}\label{sec:ex}
  Here we test our meshfree finite difference method on two examples. We
  demonstrate the convergence rates of the method, and compare our method with
  that of \cite{froese_meshfree_2017}.

  Our code, written in Python, is publicly available at
  \url{https://github.com/cfinlay/pyellipticfd}.

  \subsection{Convex envelope}
  Our first example is the convex envelope of a function $g( x)$ on a convex domain
  $\Omega$. The convex envelope has been well
  studied. In \cite{oberman_ce_2007} it was shown that the convex envelope solves the partial differential equation
  \begin{equation}
    \begin{cases}\max\{u( x) - g( x), -\Lambda_- u( x)\} = 0 &\qquad
       x \in \Omega\\
      u( x) = g( x) &\qquad  x \in \partial \Omega,
      \label{eq:ce}
    \end{cases}
  \end{equation}
  where $\Lambda_- u( x)$ is the minimal eigenvalue of the Hessian. A stable,
  monotone convergent finite difference scheme for computing the convex envelope
  was presented in \cite{oberman_computing_2008}.

  In what follows, we take $g( x)$ to be the Euclidian distance to two points
  $ p_1$ and $ p_2$,
  \begin{equation}
    g( x) = \min_{i=1,2}\{\| x - p_i\|\},
    \label{eq:g}
  \end{equation}
  or in otherwords, a double cone.

  We start by computing the solution on the square $[-1,1]^2$, with $ p_{1,2}
  = (\pm \frac 3 7, 0)$. We discretize $\Lambda_- u( x)$ using our symmetric
  linear interpolation finite difference scheme for eigenvalues of the Hessian, presented in Section
  \ref{sec:eigs}, and using the wide stencil method developed in
  \cite{oberman_computing_2008}. We call the latter a nearest neighbour scheme.
  For both methods, we solved the equation using stencils with radius two and
  three. Figure \ref{fig:cegrid} and Table \ref{tab:ce} present convergence rates
  in the max norm. We can see that for stencil radius two, angular resolution
  error arises quickly as $h$ is decreased, and the error plateaus. However, with
  stencil radius three, we get a better handle on the convergence rate of the
  error. The standard wide stencil method achieves roughly $\mathcal
  O(h^\frac{2}{3})$, while the symmetric linear interpolation method achieves
  $\mathcal O(h)$, as expected.

  Although the convergence rate of the linear interpolation method is better than
  the nearest neighbour method, for the values of $h$ we studied, the linear
  interpolation method has higher absolute error. This is because in order to
  guarantee convergence, the linear interpolation method must choice points
  greater than the minimum search radius, whereas the standard wide stencil
  finite difference scheme may choose its nearest neighbours. Thus the spatial
  resolution error of the linear interpolation scheme is generally higher than
  the nearest neighbour scheme.

  We are also interested in the error of the schemes as a function of the angular
  resolution. To this end, for fixed $h$, we compare the error of the schemes when
  the grid has been rotated off axis. Our results are presented in Figure
  \ref{fig:ang}.  The mean of
  the error of the linear interpolation scheme is higher than the nearest
  neighbour scheme, due to the fact that the linear interpolation scheme chooses
  points further from the stencil centre. However,
  the variance of the error for the linear interpolation scheme nearest neighbour
  scheme is much less than that of the nearest neighbour scheme. That is, the
  linear interpolation scheme depends less on the angular resolution of the
  stencil relative to the rotation of the grid.

  Finally, we compare the linear interpolation scheme with Froese's scheme on the
  unit disc, using an irregular triangulation of points. We generate the interior points using
  the triangulation software DistMesh \cite{distmesh}, and augment the boundary
  with additional points to ensure a sufficient boundary resolution. Convergence
  rates are presented in Figure \ref{fig:cedisc} and Table \ref{tab:ce}. We can
  see that the linear interpolation scheme achieves both the best rate of
  convergence and a better absolute error.

  \begin{figure}[t]
    \centering
    \begin{subfigure}[b]{0.911\textwidth}
      \includegraphics[width=\textwidth,left]{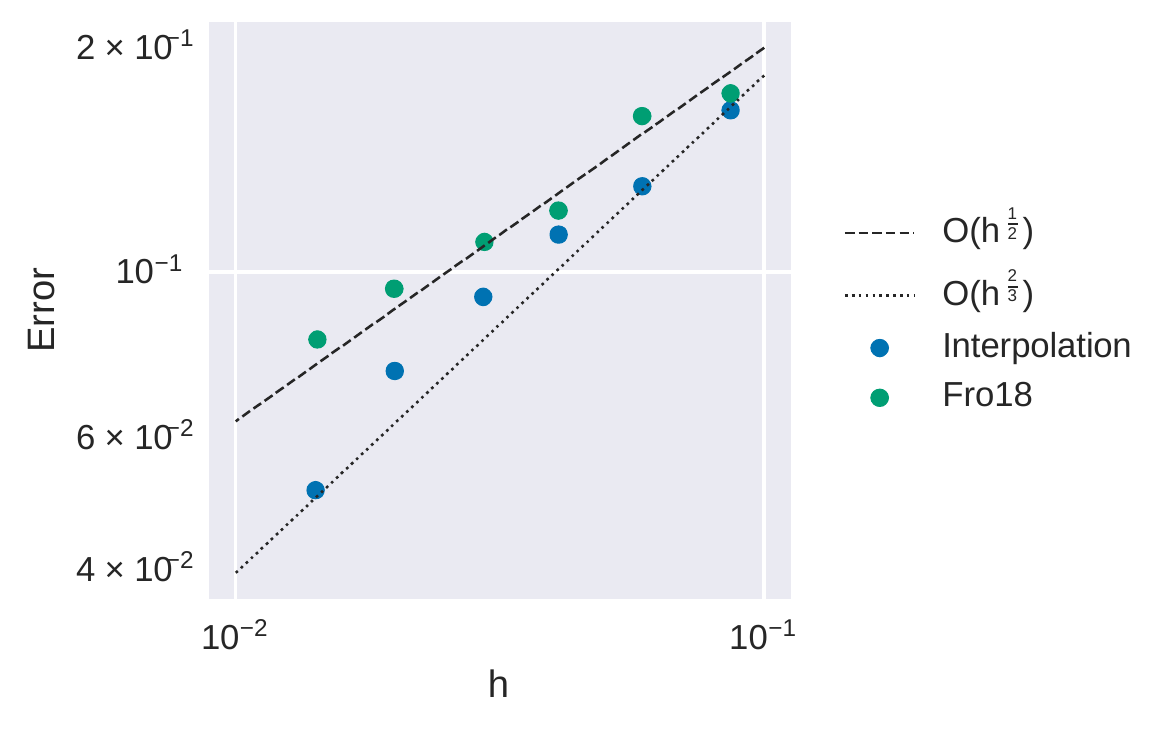}
      \caption{Triangular mesh}
      \label{fig:cedisc}
    \end{subfigure}
    \hfill
    \begin{subfigure}[b]{1\textwidth}
      \includegraphics[width=\textwidth,center]{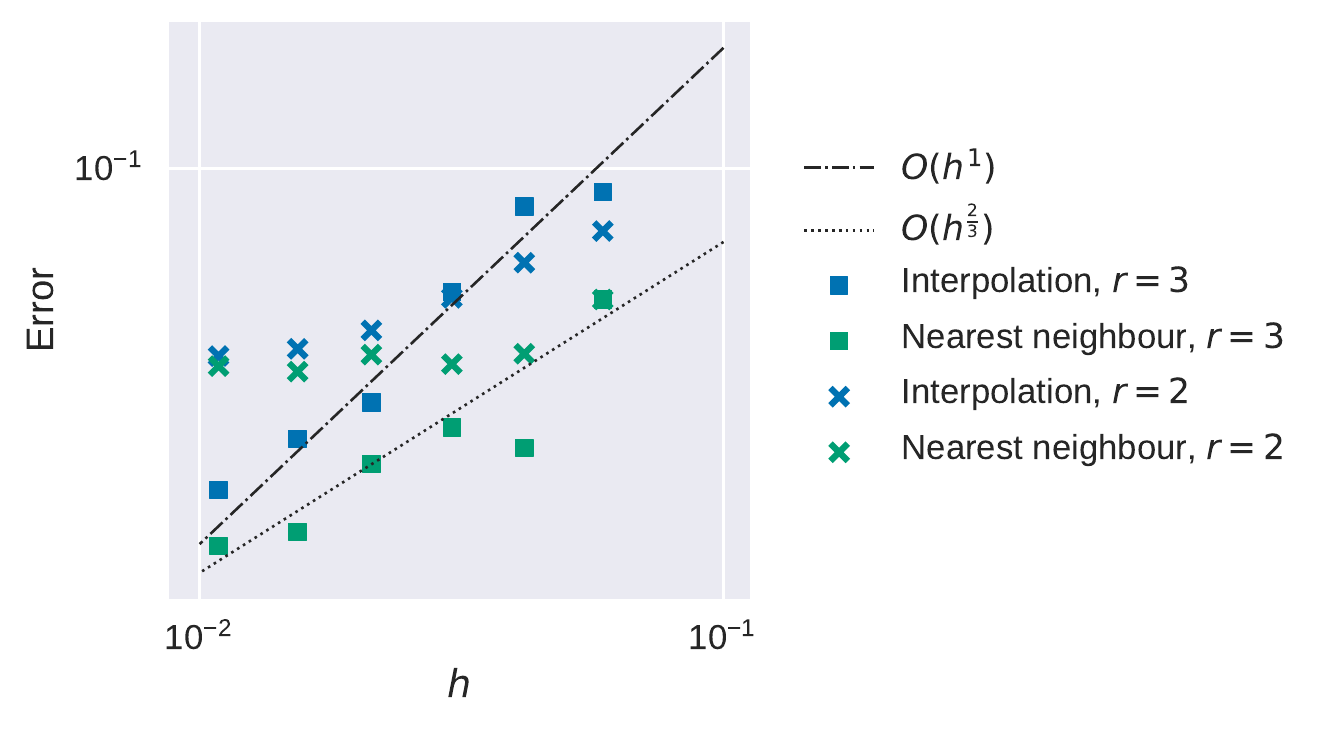}
      \caption{Regular grid}
      \label{fig:cegrid}
    \end{subfigure}
    \caption{Figure \ref{fig:cedisc}: Convergence plot for the convex envelope on
      the unit disc with
      triangular mesh.  Figure \ref{fig:cegrid}: Convergence plot for the convex
      envelope on a regular grid over
    the square $[-1,1]^2$.}
    \label{fig:ce_rates}
  \end{figure}

  \begin{table}
  \centering
  \begin{tabular}{l r l l}
    \hline
    \hline
    \multicolumn{4}{l}{Triangular mesh, interpolation} \\
        $h$    &   $N$    & Error    &  rate \\
    \hline
    \num{8.6e-2} &  427 & 0.16 &     -- \\
    \num{5.9e-2} &  785 & 0.13 & 0.61 \\
    \num{4.0e-2} & 1452 & 0.11 & 0.41 \\
    \num{2.9e-2} & 2713 & 0.09 & 0.58 \\
    \num{2.0e-2} & 5101 & 0.07 & 0.59 \\
    \num{1.4e-2} & 9674 & 0.05 &  1.07 \\
    \hline
    \hline

    \multicolumn{4}{l}{Regular grid, interpolation, $r=2$} \\
        $h$    &   $N$    & Error    &   rate \\
    \hline
    \num{5.9e-2} &  392 & \num{7.5e-2} &      -- \\
    \num{4.1e-2} &  721 & \num{6.5e-2} &  0.43 \\
    \num{3.0e-2} & 1288 & \num{5.5e-2} &  0.51 \\
    \num{2.1e-2} & 2492 & \num{4.7e-2} &  0.43 \\
    \num{1.5e-2} & 4616 & \num{4.3e-2} &  0.26 \\
    \num{1.0e-2} & 9017 & \num{4.2e-2} & 0.09 \\
    \hline
    \hline

    \multicolumn{4}{l}{Regular grid, interpolation, $r=3$} \\
        $h$    &   $N$    & Error    &  rate \\
    \hline
    \num{5.9e-2} &  528 & \num{9.0e-2} &     -- \\
    \num{4.2e-2} &  913 & \num{8.4e-2} & 0.20\\
    \num{3.0e-2} & 1552 & \num{5.6e-2} &  1.24\\
    \num{2.1e-2} & 2868 & \num{3.4e-2} &  1.45\\
    \num{1.5e-2} & 5136 & \num{2.9e-2} & 0.52\\
    \num{1.0e-2} & 9757 & \num{2.2e-2} & 0.68\\
    \hline
  \end{tabular}
  \hspace{1em}
  \begin{tabular}{l r l l}
    \hline
    \hline

    \multicolumn{4}{l}{Triangular mesh, Fro17} \\
        $h$    &    $N$    & Error    &  rate \\
    \hline
    \num{8.6e-2} &   427 & 0.17 &     -- \\
    \num{5.0e-2} &   810 & 0.16 & 0.18 \\
    \num{4.1e-2} &  1533 & 0.12 & 0.80 \\
    \num{3.0e-2} &  2908 & 0.11 & 0.30 \\
    \num{2.0e-2} &  5526 & 0.09 & 0.37 \\
    \num{1.4e-2} & 10542 & 0.08 & 0.47 \\
    \hline
    \hline

    \multicolumn{4}{l}{Regular grid, Nearest neighbour, $r=2$} \\
        $h$    &   $N$    & Error    &   rate \\
    \hline
    \num{5.9e-2} &  392 & \num{5.4e-2} &      -- \\
    \num{4.2e-2} &  721 & \num{4.2e-2} &  0.73 \\
    \num{3.0e-2} & 1288 & \num{4.0e-2} &   0.15 \\
    \num{2.1e-2} & 2492 & \num{4.2e-2} & -0.13 \\
    \num{1.5e-2} & 4616 & \num{3.9e-2} &  0.24 \\
    \num{1.0e-2} & 9017 & \num{4.0e-2} & -0.07 \\
    \hline
    \hline

    \multicolumn{4}{l}{Regular grid, Nearest neighbour, $r=3$} \\
        $h$    &   $N$    & Error    &  rate \\
    \hline
    \num{5.9e-2} &  528 & \num{5.4e-2} &     -- \\
    \num{4.2e-2} &  913 & \num{2.7e-2} &  2.0\\
    \num{3.0e-2} & 1552 & \num{3.0e-2} & -0.29\\
    \num{2.1e-2} & 2868 & \num{2.5e-2} & 0.47\\
    \num{1.5e-2} & 5136 & \num{1.9e-2} & 0.98\\
    \num{1.0e-2} & 9757 & \num{1.7e-2} & 0.18\\
    \hline
  \end{tabular}

  \caption{Errors and convergence order for the convex envelope.}
  \label{tab:ce}
\end{table}

  \begin{figure}[t]
    \centering
    \includegraphics[width=1\textwidth]{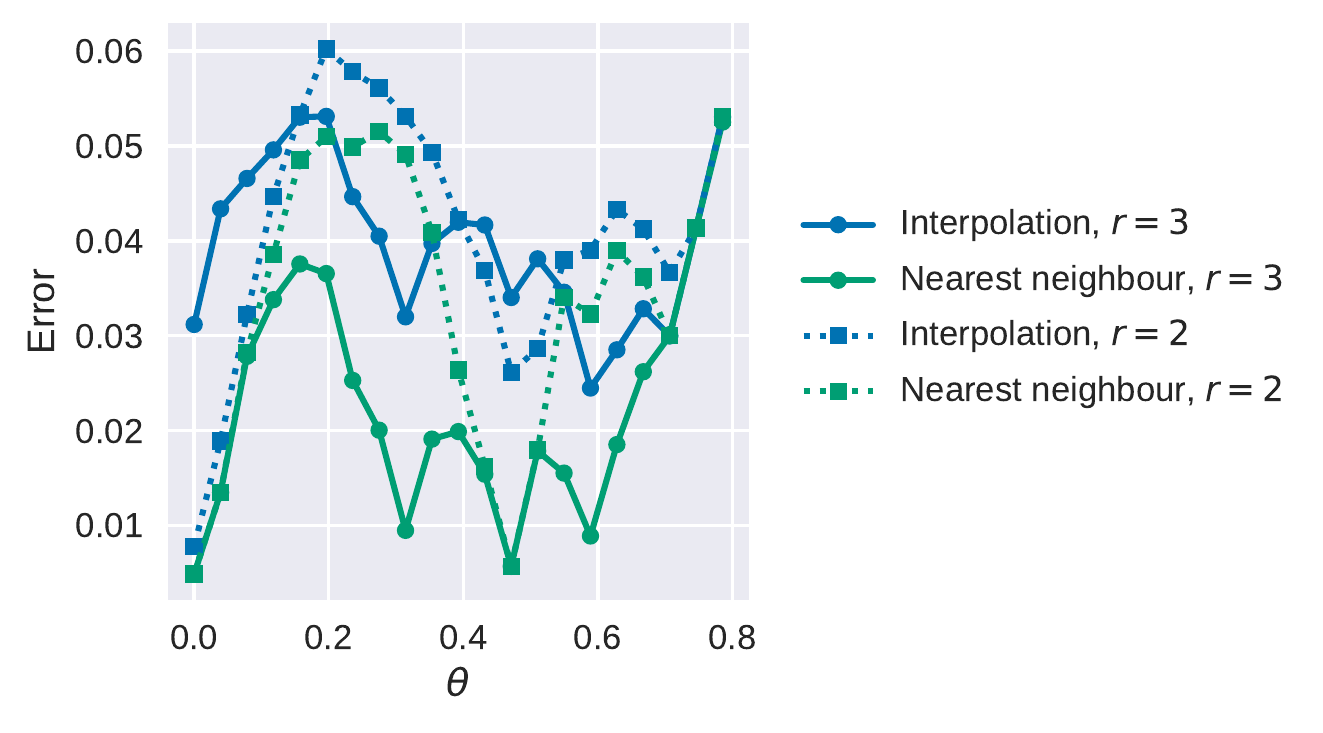}
    \caption{Error of the numerical solutions of the convex envelope PDE on a
    regular grid, as a function of rotation of the grid.}
    \label{fig:ang}
  \end{figure}

  \subsection{Pucci equation}
  Our next example is the Pucci equation,
  \begin{align}
    \begin{cases}
      \alpha \Lambda_+u(x) + \Lambda_- u(x) = 0 \qquad & x \in \Omega\\
      u(  x) = g( x) \qquad & x \in  \partial \Omega
    \end{cases}
    \label{eq:pucci}
  \end{align}
  where $\alpha$ is a positive scalar, and $\Lambda_- u$ and $\Lambda_+ u$ are
  respectively the minimal and maximal eigenvalues of the Hessian.  A convergent,
  monotone and stable finite difference scheme for the Pucci equation was first
  developed in \cite{oberman_wide_2008}. Following
  \cite{dean2005numerical,oberman_wide_2008}
  we take
  \begin{align}
    u(x,y) = -\rho^{1-\alpha},\quad \rho(x,y) = \sqrt{(x+2)^2 + (y+2)^2}.
    \label{eq:puc_ex}
  \end{align}
  We compute solutions on  the unit disc and the square $[-1,1]^2$.

  We discretized the square using a regular grid, and use either the nearest neighbour
  scheme, or the symmetric finite difference interpolation scheme presented in
  Section \ref{sec:eigs}. We use stencils of radius two or three. Errors and rates
  of convergence on the grid are presented in Table \ref{tab:pu} and in Figure
  \ref{fig:pugrid}. Both methods achieve roughly the same convergence rate before
  angular resolution error dominates. The nearest neighbour scheme achieves a
  slightly better error rate.

  As in the convex envelope example, we used DistMesh to triangulate the unit
  disc. Error and convergence rates are shown Table \ref{tab:pu} and Figure
  \ref{fig:pudisc}. Both methods achieve nearly $\mathcal O(1)$ convergence rate,
  which is better than predicted by our analysis. We hypothesize this is due to
  the fact that this example is smooth on the domain studied.
  \begin{figure}[t]
    \centering
    \begin{subfigure}[b]{0.911\textwidth}
      \includegraphics[width=\textwidth,left]{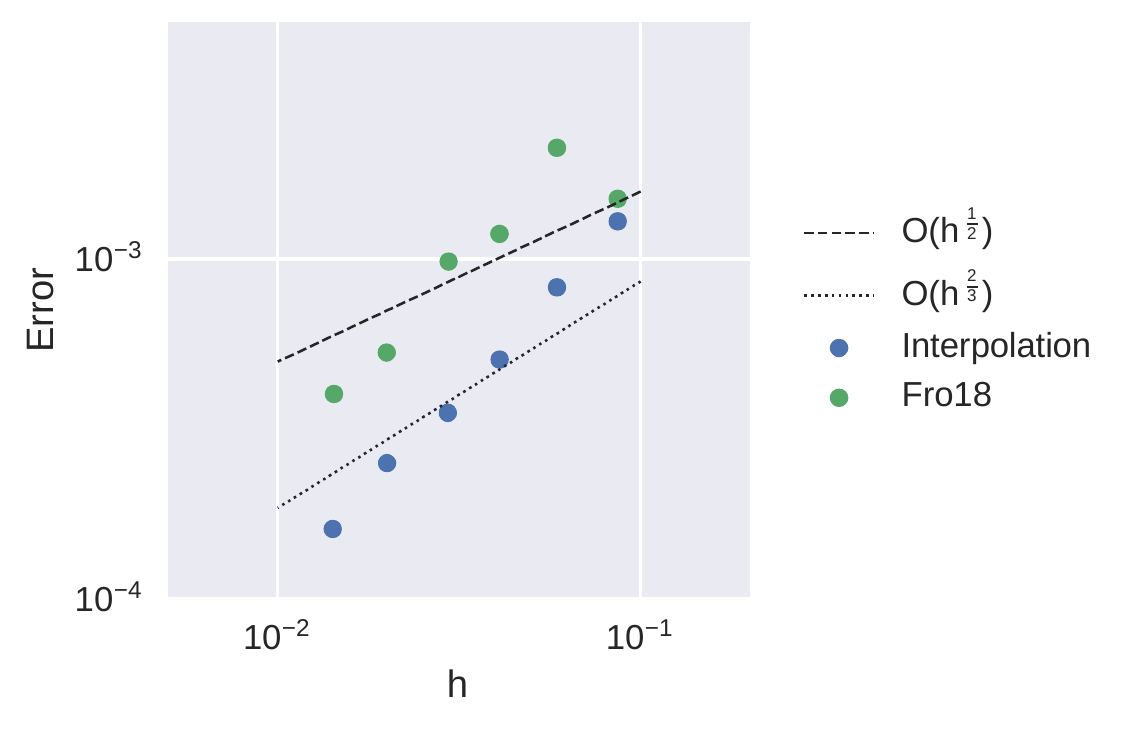}
      \caption{Triangular mesh}
      \label{fig:pudisc}
    \end{subfigure}
    \hfill
    \begin{subfigure}[b]{1\textwidth}
      \includegraphics[width=\textwidth,center]{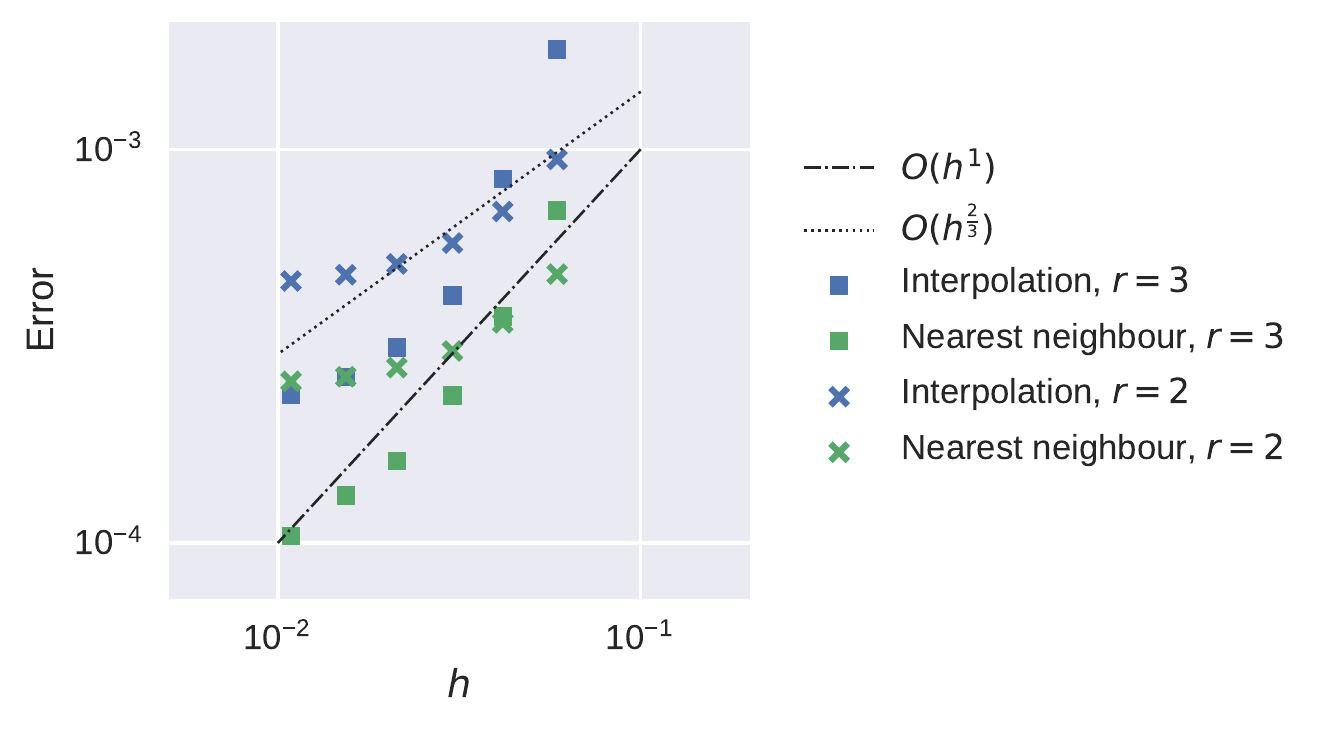}
      \caption{Regular grid}
      \label{fig:pugrid}
    \end{subfigure}
    \caption{Figure \ref{fig:pudisc}: Convergence plot for the Pucci equation on
      the unit disc with
      triangular mesh.  Figure \ref{fig:pugrid}: Convergence plot for the Pucci
      equation  on a regular grid over
    the square $[-1,1]^2$.}
    \label{fig:pu_rates}
  \end{figure}

  \begin{table}
  \centering
  \begin{tabular}{l r l l}
    \hline
    \hline
    \multicolumn{4}{l}{Triangular mesh, interpolation} \\
        $h$    &   $N$    & Error    &  rate \\
    \hline
    \num{8.6e-2}  &  427 & \num{1.3e-3} &       --  \\
    \num{5.0e-2}  &  785 & \num{8.3e-4} &  1.16  \\
    \num{4.1e-2}  & 1452 & \num{5.0e-4} &  1.34  \\
    \num{3.0e-2}  & 2713 & \num{3.5e-4} &  1.09 \\
    \num{2.0e-2}  & 5101 & \num{2.5e-4} &  0.88\\
    \num{1.4e-2}  & 9674 & \num{1.6e-4} &  1.29\\
    \hline
    \hline

    \multicolumn{4}{l}{Regular grid, interpolation, $r=2$} \\
        $h$    &   $N$    & Error    &  rate \\
    \hline
    \num{5.9e-2} &  392 & \num{9.4e-4}  &       --  \\
    \num{4.1e-2} &  721 & \num{7.0e-4}  & 0.88  \\
    \num{3.0e-2} & 1288 & \num{5.8e-4}  & 0.58  \\
    \num{2.1e-2} & 2492 & \num{5.1e-4}  & 0.35  \\
    \num{1.5e-2} & 4616 & \num{4.8e-4}  & 0.19  \\
    \num{1.0e-2} & 9017 & \num{4.6e-4}  & 0.11  \\
    \hline
    \hline

    \multicolumn{4}{l}{Regular grid, interpolation, $r=3$} \\
        $h$    &   $N$    & Error    &   rate \\
    \hline
    \num{5.9e-2}  &  528 & \num{1.e-3} &       --  \\
    \num{4.1e-2}  &  913 & \num{8.4e-4} & 2.20\\
    \num{3.0e-2}  & 1552 & \num{4.2e-4} & 2.14\\
    \num{2.1e-2}  & 2868 & \num{3.1e-4} & 0.86\\
    \num{1.5e-2}  & 5136 & \num{2.6e-4} & 0.53\\
    \num{1.0e-2}  & 9757 & \num{2.3e-4} & 0.30\\
    \hline
  \end{tabular}
  \hspace{1em}
  \begin{tabular}{l r l l}
    \hline
    \hline

    \multicolumn{4}{l}{Triangular mesh, Fro17} \\
        $h$    &    $N$    & Error    &  rate \\
    \hline
    \num{8.6e-2} &   427&  \num{1.5e-3}&        -- \\
    \num{5.0e-2} &   810&  \num{2.1e-3}& -0.90\\
    \num{4.1e-2} &  1533&  \num{1.2e-3}&   1.60\\
    \num{3.0e-2} &  2908&  \num{9.9e-4}&  0.58\\
    \num{2.0e-2} &  5526&  \num{5.3e-4}&   1.57\\
    \num{1.4e-2} & 10542&  \num{4.0e-4}&  0.841\\
    \hline
    \hline
    \multicolumn{4}{l}{Regular grid, Nearest neighbour, $r=2$} \\
        $h$    &   $N$    & Error    &  rate \\
    \hline
    \num{5.9e-2}  &  392 & \num{4.8e-4} &        -- \\
    \num{4.1e-2}  &  721 & \num{3.6e-4} &  0.83\\
    \num{3.0e-2}  & 1288 & \num{3.0e-4} &  0.52\\
    \num{2.1e-2}  & 2492 & \num{2.8e-4} &  0.28 \\
    \num{1.5e-2}  & 4616 & \num{2.6e-4} &  0.16\\
    \num{1.0e-2}  & 9017 & \num{2.6e-4} &  0.07\\
    \hline
    \hline

    \multicolumn{4}{l}{Regular grid, Nearest neighbour, $r=3$} \\
        $h$    &   $N$    & Error    &   rate \\
    \hline
        \num{5.9e-2} &  528 & \num{7.0e-4}  &       --  \\
        \num{4.1e-2} &  913 & \num{3.8e-4}  &  1.80  \\
        \num{3.0e-2} & 1552 & \num{2.4e-4}  &  1.45  \\
        \num{2.1e-2} & 2868 & \num{1.6e-4}  &  1.08  \\
        \num{1.5e-2} & 5136 & \num{1.3e-4}  & 0.62\\
        \num{1.0e-2} & 9757 & \num{1.0e-4}  &  0.68  \\
    \hline
  \end{tabular}

  \caption{Errors and convergence order for the Pucci equation.}
  \label{tab:pu}
\end{table}

  \subsubsection{Solver comparison}
  Finally, we performed a comparison of the three solvers (semi-smooth Newton, Euler, and a
  combination of the two) in terms of CPU time, for the Pucci equation on a
  regular grid. Results are presented in Table \ref{tab:wc}.

  As a function of number of grid points, the CPU time of Euler's method is roughly $\mathcal
  O(N^2)$ for both methods, interpolation and nearest neighbour. For the
  interpolation finite different schemes, both semi-smooth Newton and the
  combination solver is nearly $\mathcal O(N)$: we calculated a log-log line of
  best fit, and found semi-smooth Newton and the combination solver to be about
  $\mathcal O(N^{1.2})$.

  Of all solvers and finite difference methods, the nearest neighbour finite
  difference scheme with the combination solver achieves the best CPU time,
  followed by the semi-smooth Newton. However, as a function of number of grid
  points, the CPU time is roughly $\mathcal O(N^{1.75})$. This rate is worse than the
  interpolation finite different scheme, and so we expect on even larger grids,
  eventually the interpolation finite difference method would be faster with
  either semi-smooth Newton or the combination solver.

  \begin{figure}[t]
    \centering
    \begin{subfigure}[b]{0.48\textwidth}
      \includegraphics[width=\textwidth]{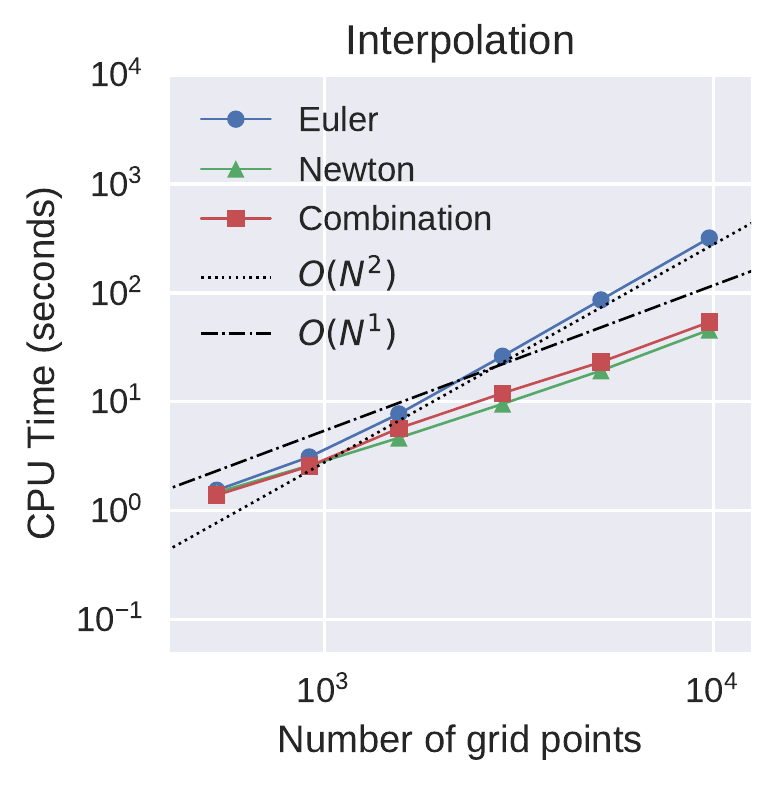}
      \label{fig:solver-i}
    \end{subfigure}
    \hfill
    \begin{subfigure}[b]{0.45\textwidth}
      \includegraphics[width=\textwidth]{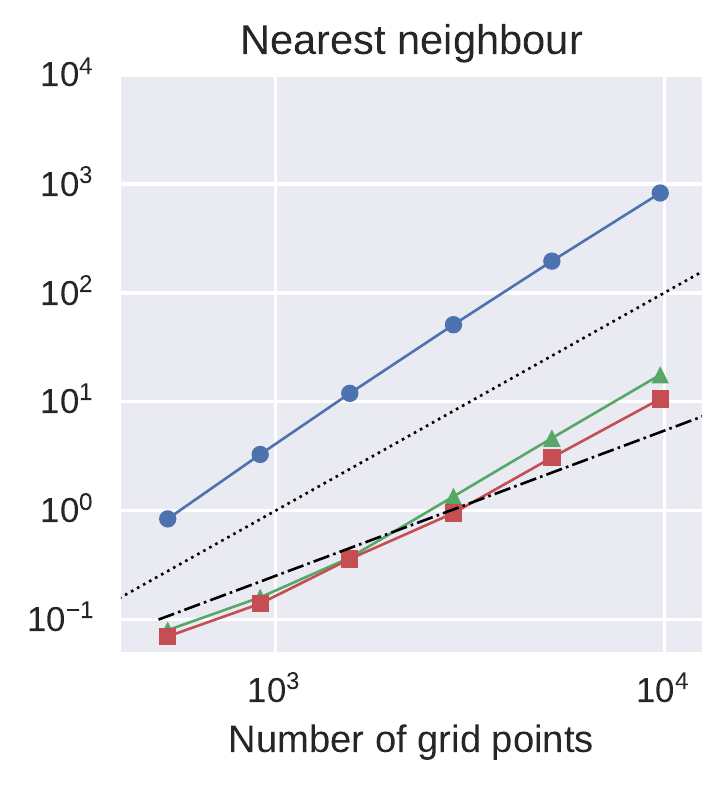}
      \label{fig:solver-nn}
    \end{subfigure}
    \caption{CPU time taken to compute solution of the Pucci equation on a regular
    grid, with stencil width $r=3$, for both methods.}
    \label{fig:pu_rates}
  \end{figure}

  \begin{table}
    \centering
    \begin{tabular}{l | r r r r r r}
      \hline
      \hline
      \multicolumn{7}{l}{Interpolation, $r=2$} \\
      $N$    & 392 & 721 & 1288 & 2492 & 4616 & 9017 \\
      \hline
      Euler  & 1.16 & 3.22 & 9.43 & 34.39 & 125.49  &  500.98 \\
      Newton  & 0.74    &  1.29   &  2.70    &  5.86    &  12.32    & 28.86     \\
      Combination &  0.75   & 1.36    &  2.50    &  6.07    & 12.39     &
      28.35     \\
      \hline
      \hline
      \multicolumn{7}{l}{Nearest neighbour, $r=2$} \\
      $N$    & 392 & 721 & 1288 & 2492 & 4616 & 9017 \\
      \hline
      Euler  & 0.43 & 1.62 & 5.75 & 24.23  & 90.90  &  383.43    \\
      Newton  &  0.06   & 0.11    &  0.25    & 1.04     &  3.40    &  13.51    \\
      Combination & 0.05    & 0.10    &  0.23    &  0.73    &  2.66    &  9.54    \\
      \hline
      \hline
      \multicolumn{7}{l}{Interpolation, $r=3$} \\
      $N$    & 528  & 913 & 1552 & 2868 & 5136 & 9757 \\
      \hline
      Euler  & 1.54 & 3.11 & 7.73 & 26.17 & 86.05 & 317.98 \\
      Newton  & 1.47    & 2.59     & 4.66     & 9.54     & 19.29     &  45.68    \\
      Combination & 1.39    & 2.56    &  5.70    &  11.97    & 23.12     &
      53.77     \\
      \hline
      \hline
      \multicolumn{7}{l}{Nearest neighbour, $r=3$} \\
      $N$    & 528  & 913 & 1552 & 2868 & 5136 & 9757 \\
      \hline
      Euler  & 0.84 & 3.28 &  11.93 & 50.90 & 195.14 & 823.99 \\
      Newton  & 0.08    &  0.16   &  0.37    &  1.35    & 4.63     &  17.66    \\
      Combination & 0.07    &  0.14   &  0.36    & 0.95     & 3.07     &
      10.61    \\
      \hline
    \end{tabular}
    \caption{Comparison of wall clock time of solvers for the Pucci equation \eqref{eq:pucci} in two
      dimensions on a regular grid. Time is reported in seconds. Results
    are for stencils of either radius $r=2$ or $r=3$.}
    \label{tab:wc}
  \end{table}

  \bibliographystyle{alpha}
  \bibliography{refs}

  \end{document}